\documentclass[reqno]{amsart}
\usepackage{amsmath,amssymb,amscd}
\usepackage{url,hyperref}
\usepackage[percent]{overpic}
\usepackage[centering,textwidth=6in,textheight=9in]{geometry}

\newtheorem{thm}{Theorem}[section]

\newtheorem{lem}[thm]{Lemma}
\newtheorem{pro}[thm]{Proposition}

\theoremstyle{definition}
\newtheorem{definition}[thm]{Definition}
\newtheorem{remark}[thm]{Remark}

\numberwithin{equation}{section}

\def\bR{\mathbb{R}}
\def\bT{\mathbb{T}}

\def\cD{\mathcal{D}}
\def\cG{\mathcal{G}}
\def\cJ{\mathcal{J}}

\def\cR{\mathcal{R}}

\def\cU{\mathcal{U}}
\def\cV{\mathcal{V}}
\def\cW{\mathcal{W}}

\def\ds{\displaystyle}

\begin{document}

\title[]{Homoclinic intersections for geodesic flows on convex spheres}

\author[Z. Xia]{Zhihong Xia}
\address{Department of Mathematics,
Southern University of Science and Technology,
Shenzhen, Guangdong, China 518055;
permanent address: Department of Mathematics, Northwestern University, Evanston, IL 60208}
\email{xia@math.northwestern.edu}

\author[P. Zhang]{Pengfei Zhang}
\address{Department of Mathematics,
University of Mississippi,
Oxford, MS 38677}
\email{pzhang2@olemiss.edu}



\begin{abstract}
In this paper, we study some generic properties
of the geodesic flows on a convex  sphere.
We prove that, $C^r$ generically ($2\le r\le\infty$),
every hyperbolic closed geodesic on $S^2$
admits some  transverse homoclinic intersections.
\end{abstract}

\maketitle

\section{Introduction}\label{sec-intro}

Let $f:M\to M$ be a  diffeomorphism on a closed manifold $M$,
$p$ be a \emph{hyperbolic} periodic point of period $n$,
and  $W^{s,u}(p)$ be the stable and  unstable
manifolds of $p$.
 A point $x\in W^s(p)\cap W^u(q)$
is called a  \emph{heteroclinic intersection} (a  \emph{homoclinic intersection} if $q=p$).
The intersection $W^s(p)\cap W^u(q)$ at $x$ is said to be  \emph{transverse}
if $T_xW^s(p)+ T_xW^u(q)=T_xM$.
Transverse homoclinic intersections play an important role in the birth of
the concept {\it Dynamical Systems}.
The complexity caused by
transverse homoclinic intersections made Poincar\'e once to believe
such phenomenon cannot exist among the solutions of the $n$-body problem
(see \cite{Dia}, or Page 167 of \cite{Mei}).
Later in \cite{Poin}, Poincar\'e described the first geometric picture
that every transverse homoclinic intersection is accumulated by
infinitely many other homoclinic intersections, and this mechanism generates
various complicated dynamical behaviors.
The study of transverse homoclinic intersections
was developed by Birkhoff  \cite{Bir35} and
then by Smale \cite{Sma65}.
The geometric model,
now called {\it Smale horseshoe}, gives a symbolic coding of the dynamics around
a transverse homoclinic intersection, and paves the way for
a systematic study of  dynamical systems with some hyperbolicity.

In this paper, we study the  existence of homoclinic intersections
 for the geodesic flows on $2$-sphere $S^2$.
Given a Riemannian metric $g$ on $S^2$,
let $K_g(x)$ be the Gauss curvature at $x$ for each $x\in S^2$.
\begin{definition}
A Riemannian metric $g$ on $S^2$ is said to be {\it convex},
if its Gauss curvature satisfies $K_g(x)>0$ for all $x\in S^2$.
Then a sphere $(S^2,g)$ is said to be {\it convex} if the metric $g$ is convex.
\end{definition}

\begin{remark}
Weyl conjectured in 1916 that for each convex sphere $(S^2,g)$,
there exists a smooth and isometric embedding $(S^2,g)\hookrightarrow \bR^3$.
This conjecture was proved independently by
Pogorelov \cite{Pog} and Nirenberg \cite{Nir}.
Therefore, each convex sphere $(S^2,g)$ is the boundary of a smooth convex
body in $\bR^3$.
\end{remark}

For each $2\le r\le \infty$,
let $\cG^r$ be the set of $C^r$-smooth Riemannian metrics on
the sphere  $S^2$,
and $\cG^r_+\subset \cG^r$ be the subset of convex metrics $g$ on $S^2$.
Endowed with $C^r$ topology, both sets $\cG^r$ and $\cG^r_+$ are  Baire spaces.

Given a Riemannian metric $g$ on $S^2$, let $\phi_t^g$ be the induced geodesic
flow on the tangent bundle $TS^2$. Note that the geodesic flow preserves
the length of tangent vectors. So one usually considers the restriction of $\phi_t^g$
on the unit tangent bundle $M_g:=\{(x,v)\in TS^2:\|v\|_{g(x)}=1\}$.

The  dynamical properties of geodesic flows
on $S^2$ have been studied extensively.
The simple topology of $S^2$ leads one to believe that
the geodesic flows on $S^2$ should be  non-chaotic. Very surprisingly,
Donnay constructed in \cite{Don88} a smooth metric on $S^2$
whose geodesic flow is {\it nonuniformly hyperbolic} and
hence extremely chaotic.
For example, this geodesic flow  even has positive metric entropy.
See \cite{BW} for an analytic example.
Note that these chaotic metrics have negative curvature almost everywhere (except three
small caps) on $S^2$.

In \cite{KW94} Knieper and Weiss constructed the first
convex sphere whose geodesic flow has positive topological entropy.
In \cite{Don} Donnay proved that a small perturbation of the standard metric
of an ellipsoid  creates transverse homoclinic intersections.
Then in \cite{CBP} Contreras and Paternain
proved the $C^2$-denseness of  metrics on $S^2$  with positive topological entropy.
Knieper and Weiss \cite{KW02}
prove the $C^\infty$-denseness of convex  metrics
on $S^2$  with positive topological entropy.
They also observed that the convexity assumption can be removed
by using a result in \cite{HWZ}. The following theorem improves the characterizations
of the geodesic flow on a generic convex sphere:

\begin{thm}\label{main}
There is a residual subset $\cR^r\subset \cG^r_+$, such that for each $g\in \cR^r$,
the geodesic flow $\phi_t^g$ on the unit tangent bundle of $S^2$
satisfies the following properties:
\begin{enumerate}
\item every closed geodesic is either hyperbolic or irrationally elliptic \cite{Abr,Ano};

\item  the stable and unstable manifolds of two hyperbolic closed geodesics admit some
transverse  intersections whenever they intersect.

\item every hyperbolic closed geodesic admits some
transverse homoclinic intersections.
\end{enumerate}
\end{thm}

The last item in the above theorem
is about the existence of homoclinic intersections for
{\it all} hyperbolic closed geodesics. This is one of the two properties
suggested by Poincar\'e \cite{Poin}:
a generic map $f\in\mathrm{Diff}^r_\mu(M)$ may satisfy the following:
\begin{enumerate}
\item[(P1)]   (hyperbolic) periodic points are  \emph{dense} in the space $M$;

\item[(P2)] for every hyperbolic periodic point $p$,
\begin{enumerate}
\item[(P2a)] $W^s(p)\cap W^u(p)\backslash \{p\}\neq\emptyset$ (weak version);

\item[(P2b)] $W^s(p)\cap W^u(p)$ is  \emph{dense} in $W^s(p)\cup W^u(p)$ (strong version).
\end{enumerate}
\end{enumerate}
These properties are closely related to the
\emph{Closing  problem} and  \emph{Connecting problem},
and have been one of the main motivations
for the recent development in dynamical systems.
Both parts have been proved when $r=1$ (see \cite{Pu,PuRo} and \cite{Tak,Xia96}).

There are some partial results when $r>1$, and most of them are about
{\it area-preserving diffeomorphisms} of closed surfaces.
Pixton \cite{Pix}, extending a result of
Robinson  \cite{Rob73},
proved (P2a) when $M=S^2$.
That is, for a $C^r$ generic area-preserving diffeomorphism on $S^2$,
every hyperbolic periodic point $p$ admits some homoclinic intersections.
Oliveira showed in \cite{Oli1} the generic existence of homoclinic  intersections of
area-preserving diffeomorphisms on $\mathbb{T}^2$.
There are some partial results on closed surfaces other than $S^2$ or $\bT^2$,
see \cite{Oli2} and \cite{Xia2} for more details.

{\it Closing and connecting problems for geodesic flows} were raised by Pugh
and Robinson in \cite{PuRo}.
The main difficulty, as described in \cite[\S 10]{PuRo},
 is that one cannot perturb the geodesic flow directly:
any perturbation of the geodesic flow has to be done through
the deformation of the metric on the manifold;
but a local deformation of the metric on the manifold
changes the dynamics of all geodesics passing
through that region. Therefore, the dynamical
effect of the perturbation of the geodesic flow is unavoidably large.
It turns out that the closing and connecting problems for geodesic flows
are still open even in the $C^1$ category (that is, among $C^2$ metrics),
50 years after Pugh's proof of  $C^1$ Closing Lemma in 1967.
Similar difficulty appears in the study of
generic properties of dynamical billiards.

Recently there is some progress in this direction.
In \cite{XZ,Z}, we proved that (P2a) holds for generic convex billiards.
In \cite{Irie} Irie proved that generically,
closed geodesics are dense on any surface\footnote{Note that there is a subtle
difference between (P1) and Irie's result, since the
phase space  for geodesic flow is the tangent space.}.
Our main aim in this paper is to show  that (P2a) holds for geodesic flows
induced by  convex metrics on $S^2$, see Theorem \ref{main}.

\subsection*{Summary of the proof} Given a convex sphere $(S^2,g)$,
the geodesic flow $\phi_t$ on $TS^2$ preserves the unit tangent bundle
and hence is a 3D flow.
Generically, each closed geodesic is either elliptic or hyperbolic (see \S \ref{sec-bum}).
We make a perturbation for each elliptic closed geodesic  such that
the rotation number of that geodesic is Diophantine
(see Proposition \ref{per-tr} and Remark \ref{rem-dio}).
To this end, we need a precise control of the effect of
of the perturbations. So we introduce the Jacobi fields
and  formulate the tangent map of $\phi_t$ in terms of the Jacobi fields (see \S \ref{sec-jac}).
Then we apply a theorem of Herman
to get the nonlinear stability of the elliptic closed geodesics  (see \S \ref{sec-her}).
After we are done with the elliptic geodesics, we move on to deal with the hyperbolic ones.
We apply a theorem of Donnay (\S \ref{sec-tra}) to destroy the non-transverse intersections
of the  stable and unstable manifolds (see Proposition \ref{opendense1}).
For the geodesic flow on a convex sphere,
we use the fact  that there exists a simple closed geodesic $\gamma_g$
(see \S \ref{sec-shortest}),
and the 3D geodesic flow $\phi_t$ can be reduced to
a 2D map $F_g$ on an annulus $A_g=\gamma_g\times (0,\pi)$ (see \S \ref{sec-bir}).
Then we   apply a theorem of Mather to deduce that
all branches of the stable and unstable manifolds are recurrent.
Finally, we prove the existence of homoclinic intersections
by counting the intersection number of two
simple closed curves on the annulus $A_g$ (see Lemma \ref{PnF}).

\begin{remark}\label{convexuse}
The convexity assumption is not essential in the proof of
Theorem \ref{KS}, which states that Kupka--Smale type properties hold
for the geodesic flow on a generic convex sphere.
Comparing with Poincar\'e's problems (P1) and (P2), we can see
a subtle difference: (P1) is about the existence/denseness of periodic points,
(P2) is about the existence/denseness of homoclinic intersections,
while the Kupka--Smale properties give {\it no} implication on the existence at all.
For example, a map $f$ satisfies  the Kupka--Smale properties
as long as it has no periodic point.
The key property for us to prove the existence of homoclinic intersections
for the geodesic flow on a generic convex sphere is a {\it dimension reduction}
from the 3D geodesic flow to a 2D map on the Birkhoff section--an annulus,
where the convexity assumption is essentially used
(see  \S \ref{sec-shortest} and  \S \ref{sec-bir}).
It is an open question whether there exists a metric $g$ on $S^2$
such that the geodesic flow induced by $g$ does not admit any Birkhoff section,
see \cite{Ban2}.
\end{remark}

\section{Preliminaries}\label{sec-pre}

Let $M$ be a closed 3D manifold, $V$ be a nonsingular vector field on $M$,
that is, $V(x)\neq 0$ for any $x\in M$, and $\phi_t$ be the flow induced by $V$.
A point $x$ is said to be {\it periodic}, if the orbit through $x$ is closed: $\phi_t(x)=x$
for some time $t>0$. The period of $x$ is the minimal time $t>0$ such that $\phi_t(x)=x$.
For each closed orbit $\gamma$ one can draw a local cross-section $\Sigma$
through a point $x\in\gamma$ and define the first return map
$P=P_{\gamma,\Sigma}$, the {\it Poincare map},
from (a smaller subset of)
$\Sigma$ to itself under $\phi_t$.
Note that $P(x)=x$, and the linearization $D_xP$
acts linearly on $T_x\Sigma$.
Let $\text{Tr}(\gamma)$ be the  {\it trace} of the linear action $D_xP$.
Note that $\text{Tr}(\gamma)$ is independent of the choices of $\Sigma$
and of the choices of $x\in\gamma$.
More generally, let  $\gamma^k$ be the closed orbit
that repeats itself $k$ times.
One can study the $k$-th return map $P^k$ and define
$\text{Tr}(\gamma^k)$ to be the trace of the action $D_xP^k$.

Assume that $P$ preserves an area form on $\Sigma$.
Then the closed orbit $\gamma$ is said to be {\it  degenerate}, if $\text{Tr}(\gamma)=2$;
to be {\it  nondegenerate} if $\text{Tr}(\gamma)\neq2$. Note that each nondegenerate
closed orbit is persistent under small perturbations of the vector field $V$.
Moreover, the closed orbit $\gamma$ is said to be {\it  hyperbolic}
if $|\text{Tr}(\gamma)|>2$; to be {\it  elliptic} if $|\text{Tr}(\gamma)|<2$
and to be  {\it parabolic} if $|\text{Tr}(\gamma)|=2$.
For an elliptic closed orbit $\gamma$,
the action $D_{x}P$ is conjugate to a rotation matrix
$\ds R_\rho=\begin{bmatrix}\cos2\pi\rho & -\sin2\pi\rho\\
\sin2\pi\rho & \cos2\pi\rho\end{bmatrix}$ (for some $0<\rho<1$).
In this case, the number $\rho$ is also called the {\it  rotation number} of
the elliptic closed orbit $\gamma$.
Then a closed orbit $\gamma$ is said to be {\it irrationally elliptic}
if the rotation number of $\gamma$ is irrational.

\subsection{Geodesic flow}\label{sec-geo}
Let $X$ be a closed surface endowed with a smooth Riemannian metric $g$.
For a point $x\in X$ and a vector $v\in T_xX$, let $\gamma_{v}(t)=\exp_{x}(tv)$
be the geodesic starting at $x$ with initial velocity $v$. This induces a smooth map
on the tangent bundle
$\phi_t:TX\to TX, (x,v)\mapsto (\gamma_{v}(t),\dot \gamma_{v}(t))$, which is the
so-called  geodesic flow.
Note that $|\dot \gamma_{v}(t)|\equiv |v|$ for all $t$. 
So we can restrict the geodesic flow $\phi_t$
to the unit tangent bundle $M_g:=\{(x,v)\in TX:g_x(v,v)=1\}$.

Note that the manifold $M_g$ changes if one perturbs the metric $g$ on $X$.
To avoid this problem, we may consider the abstract sphere bundle
$M=\big(TX\backslash\{0_X\}\big)\slash \bR_+$,
where $0_X:X\to TX$ is the zero section and $\bR_+=(0,\infty)$.
Note that there is a canonical isomorphism $i_g:M_g\to M$,
and one can study the abstract geodesic flow $i_g\circ\phi_t\circ i_g^{-1}$ on $M$.
We will not distinguish these two settings.

Let $TTX$ be the double tangent bundle of $X$.
Note that the fiber $T_{(x,v)}(TX)$
can be identified with $T_xX\times T_xX$ via the map
$\xi\mapsto (d\pi(\xi),C(\xi))$, where $\pi:TX\to X$ is the natural projection
and $C:T(TX)\to TX$ is the connection map induced by the Levi-Civita connection
related to the Riemannian metric $g$.
Under this isomorphism, $T_{(x,v)}M_g$
is mapped  to $T_xX\times v^{\perp}$, where $v^{\perp}\subset T_xX$ is the subspace
orthogonal to $v$.
Moreover, the vector field $G:M_g\to TM_g$, $(x,v)\mapsto(v,0)$
generates the geodesic flow $\phi_t$ on $M_g$.

\subsection{Bumpy Metric Theorem}\label{sec-bum}
Let $\phi_t$ be the geodesic flow on
the unit tangent bundle of the closed surface $(X,g)$.
A periodic orbit of the geodesic flow $\phi_t$ corresponds
to a closed geodesic $\gamma$ on $X$, and the minimal period of
the orbit, say $T$, corresponds the prime length of the closed geodesic $\gamma$.
Note that each closed geodesic corresponds to two closed orbits
of the geodesic flow: one moves forward, and the other one moves backward.
Clearly these two have the same dynamical properties.

The Riemannian metric $g$ on $X$ is said to be {\it bumpy},
if every closed geodesic, viewed as a periodic orbit of the geodesic flow on the unit
tangent bundle,
is either hyperbolic or irrationally elliptic (that is,
the rotation number $\rho$ is irrational if  $\gamma$ is elliptic).
The following theorem was formulated by Abraham in
\cite{Abr} and  proved by Anosov in \cite{Ano}.

\noindent{\bf Bumpy Metric Theorem.}
{\it Let  $\cG^r$ be the space of $C^r$ Riemannian metrics on $X$.
Then the set of bumpy metrics on $X$ is residual in $\cG^r$.}

\subsection{Jacobi fields}\label{sec-jac}
One of the main tools in the study of geodesic flows
(especially on surfaces) are Jacobi fields,
which represent infinitesimal variations along a geodesic.
More precisely, let $\gamma:\bR\to X$ be a geodesic with unit speed,
and $\gamma_s:\bR\to X$ be a smooth family of geodesics with $\gamma_0=\gamma$.
Then $\ds J(t)=\frac{d}{ds}\Big|_{s=0}\gamma_s$ defines a vector field
along $\gamma$, a so-called Jacobi field.
A Jacobi field $J(t)$ along a geodesic $\gamma$ satisfies the Jacobi equation:
$\ddot J+R(J,\dot\gamma)\dot\gamma=0$, where $R$ is the curvature tensor.
In fact, the evolution of Jacobi fields with time is governed
by the tangent  of the geodesic flow.
That is, let $\gamma(t)$ be the geodesic with initial position $(x,v)$,
and $J(t)$ be the Jacobi field along $\gamma$ with the initial condition
$(J(0),\dot J(0))=(J_0,\dot J_0)$. Under the identification
$T_{(x,v)}TX$ with $T_xX\times T_xX$, we have
$D_{(x,v)}\phi_t(J_0,\dot J_0)=(J(t),\dot J(t))$, see \cite{Pat}.

Note that if $J$ and $\dot J$
are orthogonal to $\dot\gamma$ at some time $t=t_0$, then they
are orthogonal to $\dot\gamma$ for all times $t$.
Such Jacobi fields are
called {\it orthogonal}.
For orthogonal Jacobi fields, the Jacobi equation can be rewritten as
$\ddot{J}(t)+K(t)\cdot J(t)=0$, where $K(t)$ is the Gaussian curvature at $\gamma(t)$.
Moreover, let $N$ be the normal vector field along $\gamma$,
and let $J=f(t)\cdot N$. Then we have a scalar differential equation $\ddot{f}+K\cdot f=0$.

The orthogonal Jacobi fields lie in  a 2D subbundle $\cJ\subset TM_g$,
where $\cJ_{(x,v)}=v^{\perp}\times v^{\perp}$.
In particular, for a  closed geodesic $\gamma$ with minimal period $T$,
the tangent map $D_{(x,v)}\phi_T$ induces a linear action on the 2D space $\cJ_{(x,v)}$
of  orthogonal Jacobi fields, and this is exactly the linearized  Poincare map
$D_{(x,v)}F$ with respect to some transversal
$\Sigma$ with $T_{(x,v)}\Sigma=\cJ_{(x,v)}$.
More precisely,
let $J_1(t)=f_1(t)N_{\gamma(t)}$ be the Jacobi field
on $\gamma$ with initial condition $(f_1(0),f_1'(0))=(1,0)$.
Note that $J_1$ may not  be $T$-periodic,
although the geodesic $\gamma$ is.
Then $J_2(t)=f_2(t)N_{\gamma(t)}$ with
$\ds f_2(t):=f_1(t)\cdot\int_0^t\frac{ds}{f_1(s)^2}$ is the Jacobi field
on $\gamma$ with initial condition  $(f_2(0),f_2'(0))=(0,1)$.
Note that $f_2(t)$ is well defined at the times when
$f_1(t_0)=0$, since in this case we must have $f_1'(t_0)\neq 0$
and $\ds \lim_{t\to t_0}f_2(t)=-\frac{1}{f_1'(t_0)}$. Then
\begin{equation}\label{DT}
D_{(x,v)}F=D_{(x,v)}\phi_T|_{\cJ}
=\begin{bmatrix}f_1(T) & f_2(T) \\ f_1'(T) & f_2'(T)\end{bmatrix},
\end{equation}
\begin{equation}\label{trace}
\mathrm{Tr} (\gamma)=f_1(T)+f_2'(T)
=f_1(T)+\frac{1}{f_1(T)}+f_1'(T)\int_0^T\frac{ds}{f_1(s)^2}.
\end{equation}

\subsection{Transversal intersections}\label{sec-tra}

Donnay proved the following perturbation result in \cite{Don}.
\begin{pro}\label{don1}
Let $(X,g)$ be a closed surface, $\gamma$ and $\eta$ be two hyperbolic closed geodesics
such that $W^s(\gamma)$ and $W^u(\eta)$ admit a non-transverse
intersection. Then there is a $C^r$-small perturbation $\hat g$ of the metric $g$, such that
$W^s(\gamma, \hat g)$ and $W^u(\eta, \hat g)$ admit  a transverse intersection.
\end{pro}
In \cite{Don} the result is stated under a stronger
assumption that two components with non-transverse
intersection coincide.
This coincidence assumption is used only in \cite[Lemma 3.1]{Don} to find a point
$z=(x,v)\in W^s(\gamma)\cap W^u(\eta)$ and
an open neighborhood $U\subset X$ of $x$
such that the geodesic $\exp_x(tv)$ passes through $U$ only once.
Then Donnay made a small perturbation of the metric inside $U$ (see \cite[Lemma 3.2]{Don}).
To apply Donnay's result in our case, it suffices to show that,
along the orbit of
a non-transverse intersection $z=(x,v)\in W^s(\gamma)\cap W^u(\eta)$,
one can always find a moment $x_0=\exp_x(t_0 v)$ and a small neighborhood
$U$ of $x_0$ such that the geodesic $\exp_x(tv)$ passes through $U$ only once.

\begin{proof}
Let $\gamma$ and $ \eta$ be two hyperbolic closed geodesics on $X$,
$z=(x,v)\in W^s(\gamma)\cap W^u(\eta)$ be a non-transverse intersection,
and $c(t)=\exp_x(t v)$ be the geodesic
passing through $x$ in the direction of $v\in T_xX$.
Let $B(\gamma,\delta)$ be the $\delta$-tubular neighborhood along $\gamma$,
and $B(\eta,\delta)$ along $\eta$. Reducing $\delta$ if necessary, we assume
$B(\gamma,2\delta)\cup B(\eta,2\delta)$ does not contain the whole geodesic $c(t)$,
$t\in\mathbb{R}$.
Note that $c(t)$ will approach $\gamma$ as $t\to+\infty$ and
approach $\eta$ as $t\to-\infty$. So there exists $R\ge 1$ such that
$c(t)\in B(\gamma,\delta)\cup B(\eta,\delta)$ whenever $|t|>R$.

The finite geodesic segment $c[-R,R]$ can have at most  finitely many self-intersections.
Therefore, one can pick a point $x_0=c(t_0)$ outside
 $B(\gamma,2\delta)\cup B(\eta,2\delta)$
which is not a point of self-intersection,
and a small neighborhood $U\subset B(x_0,\delta)$ such that
$c[-R,R]$ passes through $U$ only once. Then by our choice of $R$,
the whole geodesic $c(\bR)$ passes through $U$ exactly once.
\end{proof}

Note that a transverse intersection, once created, persists under small perturbations.

\subsection{Existence of closed geodesics on $S^2$}\label{sec-shortest}
It is well known that there are infinitely many closed geodesics on any closed surface.
For surfaces  $X\neq S^2$, this is trivial since $\pi_1(X)\neq 0$ and
the curves minimizing the length/energy among the set of closed curves $c:S^1\to X$
within a given nonzero homotopy class must be geodesic.
This {\it curve shortening process}  does not work on $X= S^2$,
since any closed curve can be deformed to a point curve.
Instead of examining the closed curves one by one, we consider
the free loop space $\Lambda S^2$ on $S^2$ (or more generally,
the 1-cycle space, see \cite{CC}).
We know that $\pi_1(\Lambda S^2)\simeq\pi_2(S^2)\neq 0$.
Let $\eta$ be a closed curve on $\Lambda S^2$ such that
$[\eta]\neq 0\in \pi_1(\Lambda S^2)$.
Note that $\eta(t)\in\Lambda S^2$
is a closed curve on $S^2$  for each $t\in S^1$.
Let $\ds l(g)=\min_{[\eta]\neq 0}\max_{t\in S^1}|\eta(t)|$.
This function $l$ is also used in the proof of Proposition \ref{assign}.
Then  there exists a nontrivial closed geodesic $\gamma_g$ on $S^2$ of length $l(g)$.
This is the {\it minimax argument} introduced by Birkhoff \cite{Bir17}.
Moreover, if we assume that the sphere $(S^2,g)$ is convex, then
\begin{itemize}
\item $l(g)$ is the minimum of the lengths of closed geodesics on $S^2$,
since each closed geodesic $\gamma$ can be embedded to
a loop $\eta:S^1\to \Lambda S^2$ with $\eta(0)=\gamma$ and $|\eta(t)|\le |\gamma|$
for all $t$;

\item any closed geodesic of length $l(g)$ is simple, that is, $\gamma$
has no self-intersection on $S^2$. For example, if $\gamma$
is of figure eight on $S^2$, one can first shift $\gamma$ along the normal direction
to get a nongeodesic but shorter figure eight curve, say  $\sigma$.
Then we can embed $\sigma$ to a path $\eta$ with $|\eta(t)|\le |\sigma|<|\gamma|$.
Therefore, $|\gamma|> l(g)$.
\end{itemize}
See \cite{CC} for more details.
Note that there may be (infinitely) many closed geodesics
with the same length, and there is no canonical way to
assign to each $g$ a closed geodesic $\gamma_g$.

\subsection{The Birkhoff section and annulus map}\label{sec-bir}
Now let's assume  the sphere $(S^2,g)$ is convex,
and $\gamma_g$ be a simple closed geodesic (see \S \ref{sec-shortest}).
Then $S^2\backslash \gamma_g$ consists of two parts, say $D_{\pm}$.
Let $A_g$ be the set of unit vectors  $(x,v)\in T_{\gamma_g}S^2$
pointing to the side of $D_+$.
For each $(\gamma_g(t),v)\in A_g$, let $\theta$ be the angle measured from
$\dot\gamma_g(t)$ to $v$.
Then we can identify the set $A_g$ with the open annulus $\gamma_g\times (0,\pi)$.
Birkhoff proved in \cite{Bir}
that the geodesic flow on $TS^2$ induces a diffeomorphism on the annulus $A_g$.
More precisely,  for any $(x,v)\in A_g$ with $0<\theta<\pi$, we have
$\phi_{t(x,v)}(x,v)\in A_g$ for some continuous positive function $t(x,v)$ on $A_g$.
Let $F_{g}:A_g\to A_g$
be the first return map of $\phi$ with respect to $A_g$.
\begin{pro}[\cite{Bir}]
Let $(S^2,g)$ be a convex sphere, $\gamma_g$ be a simple closed geodesic,
$A_g$ be the set of unit vectors along $\gamma_g$ pointing
to the same side of $S^2\backslash \gamma_g$.
Then the first return map $F_g$ of the geodesic flow on $A_g$ is
a smooth diffeomorphism, and it
preserves the 2-form $\omega=\sin\theta dt\wedge d\theta$ on $A_g$.
\end{pro}
The annulus $A_g$ is the so-called  {\it Birkhoff annulus} of the geodesic flow $\phi_t$,
and the map $F_g$ is the so-called
the {\it Birkhoff annulus map} with respect to $(g,\gamma_g)$.
This reduction from a 3D flow to a 2D map on $A_g$ is the main reason
that we put the convex assumption of the Riemannian metrics on $S^2$,
and is the key property that we can apply Mather's result \cite{Mat82} in the proof
of Lemma \ref{PnF}.

As we have mentioned at the end of Remark \ref{convexuse}, it is an open question
whether there  exists a metric on $S^2$ for which the induced geodesic flow
does not admit any Birkhoff section.

\section{Perturbations of closed geodesics}\label{sec-per-tr}

Let $(X,g)$ be a closed surface,
$M_g\subset TX$ be the unit tangent bundle and $\phi_t$ be the geodesic flow on $M_g$.
Given a closed geodesic $\gamma$ on $X$,
let $\Sigma\subset M_g$ be a cross-section of the geodesic flow at some point
$(x,v)=(\gamma(t_0),\dot\gamma(t_0))$,
then $F=F_{\gamma,\Sigma}$ be the Poincare map
of the geodesic flow on $\Sigma$,
and $\text{Tr}(\gamma)=\text{Tr}(D_{(x,v)}F)$
be the trace of the linearized action.
Note that $\text{Tr}(\gamma)$ is independent of the choices $x\in\gamma$
and of the choices of transversals $\Sigma$.
Moreover, $\text{Tr}(\gamma)=f_1(T)+f_2'(T)$, see Eq. \eqref{trace}.

\begin{pro}\label{per-tr}
Let $(X,g)$ be a closed surface and
$\gamma$ be a closed geodesic.
Then there exists a $C^r$ perturbation $\hat g$
of the metric $g$, such that $\gamma$ is still a closed geodesic
for the new metric $\hat g$ and the trace
$\mathrm{Tr}_{\hat g}(\gamma)\neq \mathrm{Tr}_g (\gamma)$.
\end{pro}
Note that the above result is a weak version of Franks' Lemma for
perturbations of abstract diffeomorphisms. Franks' Lemma for geodesic flows
turns out to be very difficult to prove, and may not hold on its full generality, see \cite{Con,Vis}.
Proposition \ref{per-tr} is sufficient for our need in this paper.
\begin{proof}
Let $\gamma$ be a closed geodesic and $T>0$ be its prime period.
Note that $\gamma$ may have
finitely many self-intersections on $X$.  Such points are called the {\it multi-points}
of $\gamma$.
We fix a {\it simple}  point on $\gamma$, say
$x_0=\gamma(t_0)$, in the sense
that $\gamma(t)\neq x_0$ for any $t\in[0,T)\backslash \{t_0\}$.
Without loss of generality we assume $t_0=0$.
Pick $\epsilon>0$ small enough such that $\gamma|_{[T-2\epsilon,T]}$ consists
only of simple points of $\gamma$, and $U$ a small tubular neighborhood
of $\gamma(T-2\epsilon,T)$ such that $\gamma\cap U=\gamma|_{(T-2\epsilon,T)}$.
This set $U$ will be the support of our perturbation  of the metric $g$.

Let $J_1(t)=f_1(t)N_{\gamma(t)}$ be the Jacobi field
on $\gamma$ with initial condition $(f_1(0),f_1'(0))=(1,0)$.
Note that $(f_1(t),f_1'(t))\neq (0,0)$ for any $t$.
In the following we first assume $f_1(T)\neq 0$ and $f_1'(T)\neq 0$.
The cases that $f_1(T)f_1'(T)=0$ will be discussed at the end of the proof.

Let $h:[0,T]\to[0,1]$ be a $C^{r+2}$ small function satisfying the following conditions:
\begin{enumerate}
\item $h(t)=0$ for all $0\le t\le T-2\epsilon$ and for all $T-\epsilon\le t\le T$;

\item $h(t)> 0$  for all $T-2\epsilon< t< T-\epsilon$.
\end{enumerate}
Let $\hat f_1(t)=f_1(t)+h(t)$ and $\hat K(t)=- \frac{\hat f_1''(t)}{\hat f_1(t)}$ for each $0\le t\le T$.
It is easy to see that $\hat K$ is smooth, $T$-periodic and $C^r$-close to $K$.
Moreover, $\hat J(t)=\hat f(t)N_{\gamma(t)}$ will be a new Jacobi field
with initial condition $(\hat f_1(0),\hat f_1'(0))=(1,0)$
if $\hat K$ describes the curvature along $\gamma$ for some metric $\hat g$.
Assuming this for a moment,  we see that $\hat f_1(T)=f_1(T)$, $\hat f_1'(T)=f_1'(T)$,
\begin{align*}
\mathrm{Tr}_{\hat g} (\gamma)&=
\hat f_1(T)+\frac{1}{\hat f_1(T)}+\hat f_1'(T)\int_0^T\frac{ds}{\hat f_1(s)^2}\\
&=f_1(T)+\frac{1}{f_1(T)}+f_1'(T)\int_0^T\frac{ds}{(f_1(s)+h(s))^2},\\
\mathrm{Tr}_{\hat g} (\gamma)-\mathrm{Tr}_g (\gamma)
&=f_1'(T)\int_{T-2\epsilon}^{T-\epsilon}
\Big(\frac{1}{(f_1(s)+h(s))^2}-\frac{1}{f_1(s)^2}\Big)ds.
\end{align*}
Therefore, $\mathrm{Tr}_{\hat g} (\gamma)\neq \mathrm{Tr}_{g} (\gamma)$
and depends continuously on the function $h$.

To construct a metric  $\hat g$ such that $\hat K(t)$ is
the curvature along $\gamma(t)$,
we will use the Fermi coordinate
along $\gamma$, that is, $(t,s)\mapsto \exp_{\gamma(t)}(sN(t))$.
In this coordinate system, the metric tensor $g$ satisfies
\begin{enumerate}
\itemsep0.5em
\item $g_{11}(t,0)=1$, $g_{12}(t,s)=g_{21}(t,s)=0$ and $g_{22}(t,s)=1$;

\item $\ds \partial_s g_{11}(t,0)=0$, $\Gamma^i_{jk}(t,0)=0$ for all $i,j,k$.

\item $\ds \partial_s^2 \sqrt{g_{11}}(t,0)=-K(t)$, or equivalently, $\ds \partial_s^2 g_{11}(t,0)=-2K(t)$.
\end{enumerate}
Then the new metric tensor  $\hat g$ on $U$ is given by
\begin{enumerate}
\itemsep0.5em
\item $\hat g_{ij}(t,s)=0$ when $i\neq j$, $\hat g_{22}(t,s)=1$,

\item $\hat g_{11}(t,s)=g_{11}(t,s)-k(t)b(s)s^2$,
\end{enumerate}
where $k(t)=\hat K(t)-K(t)$, and $b$ be a smooth bump function with $b(0)=1$
and a uniform $C^r$-norm.
Then the new metric $\hat g$ is $C^r$ close to $g$ 
and is identical to $g$ on $X\backslash U$.
Note that the curve $(t,0)\mapsto \gamma(t)$ is still a closed geodesic
under $\hat g$ with unit speed, and
the new curvature at $\gamma(t)$ is
$-\frac{1}{2}\partial_s^2\hat g_{11}(t,0)=K(t)+k(t)=\hat K(t)$.

Lastly, if $f_1(T)=0$, then $f_1'(T)\neq 0$.  In this case
we use a two-step perturbation of the metric:
the first one makes  $\hat f_1(T)\neq 0$ (while keeping $\hat f_1'(T)\neq 0$),
and the second one (much smaller) changes the
trace $\mathrm{Tr}_g(\gamma)$ continuously.
The detail is omitted since the perturbations
are of the same type used above.
If $f_1'(T)=0$, then $f_1(T)\neq 0$ and we can employ the two-step process, too.
This completes the proof.
\end{proof}
\begin{remark}\label{rem-dio}
It follows from Proposition \ref{per-tr} that
\begin{itemize}
\item if $\gamma$ is degenerate, then after the perturbation, it is either hyperbolic
or elliptic;

\item if $\gamma$ is elliptic, then we can manipulate its rotation number
(irrational, Diophantine, etc).
\end{itemize}
\end{remark}

Recall that $\cG^r$ is the set of $C^r$-smooth Riemannian metrics on $S^2$,
and $\cG^r_+$ is the subset of convex Riemannian metrics on $S^2$.
 For each $g\in \cG^r_+$, let $l(g)$ be the length of the
shortest closed geodesic on the sphere $(S^2,g)$.
See Sect. \ref{sec-shortest} for more details.
\begin{pro}\label{assign}
There is an open and dense subset $\cU^r\subset \cG^r_+$
such that
\begin{enumerate}
\item for each $g\in\cU^r$, there is a unique closed geodesic of length $l(g)$,

\item the map $g\in\cU^r\mapsto\gamma_g\in C^r(S^1,S^2)$ varies smoothly.
\end{enumerate}
\end{pro}
Therefore, on an open and dense subset $\cU^r\subset \cG^r_+$,
there is a canonical choice of the simple closed geodesic $\gamma_g$,
and the induced map $F_g$ on $A_g=\gamma_g\times(0,\pi)$
also depends continuously on $g$.
\begin{proof}
Let $g_0\in\cG^r_+$ be a convex Riemannian metric on $S^2$.
Then $a:=\min K_0(x)>0$. Let $b:=\max K_0(x)$,
and $\cU\subset\cG^r_+$ be a small neighborhood of $g_0$ such that
for each $g\in\cU$, one has $a/2\le K_g(x)\le 2b$ for any $x\in S^2$.
Consider the function $l:g\in\cU\to l(g)$ (see \S \ref{sec-shortest}).
Note that there exists a constant $C\ge \pi$
such that $l(g)^2\le C\cdot \text{Area}(g_0)$ for any $g\in \cU$ (see \cite{Cro}).
For any $g_n\to g_0$ and for any closed geodesic $\gamma_n$
on the sphere $(S^2,g_n)$,
the limit of $\gamma_n$ (passing to a subsequence if necessary)
is a simple closed geodesic $\gamma$ whose length $|\gamma|\le \lim l(g_n)$. Therefore,
$l(g)\le \liminf l(g_n)$, and the function $l:g\in \cG^r_+\to l(g)$
is lower semi-continuous.
Let $\cR_l^r$ be the set of metrics of continuity of the function $l$,
which is a residual subset of $\cG^r_+$.

Let $\cR^r_b\subset \cG^r_+$ be the residual subset of bumpy metrics
such that each closed geodesic is either hyperbolic or irrationally elliptic.
Let $g\in\cR_l^r\cap \cR_b^r$.
Clearly there is only a finitely many closed geodesics of length less than $l(g)+1$.
Label them according to their lengths as $\gamma_1,\cdots,\gamma_k$,
$\gamma_{k+1},\cdots,\gamma_n$, where $|\gamma_i|=l(g)$ for each $1\le i\le k$,
and $|\gamma_j|\ge |\gamma_{k+1}|> l(g)$ for each $k<j\le n$.
For each $i=2,\dots,k$, we select a {\it simple} point $x_i$ on $\gamma_i$.
Here, a point $x$ is `simple' if the union $\bigcup_{1\le j\le n}\gamma_j$
 covers $x$ only once.
Then we perturb the metric tensor $g$ around each $x_i$
such that all $\gamma_i$, $2\le i\le k$ are longer nondegenerate geodesics for
the new metric $\hat g$:
$|\gamma_i|_{\hat g}\ge l(g)+\epsilon$
while $\gamma_{\hat g}=\gamma_1$ is unchanged.
Then there exists a very small neighborhood $\cV$ of $\hat g$,
such that for any $\tilde g\in \cV$, the continuation
$\gamma_{\tilde g}$ of $\gamma_1$ is the only
geodesic of shortest length (using the nondegeneracy condition of $\gamma_{\hat g}$),
and $\tilde g\in\cV\mapsto \gamma_{\tilde g}$ varies
continuously.
We complete the proof by letting $g$ vary in $\cR_l^r\cap \cR_b^r$.
\end{proof}

Let $A_g$ be the Birkhoff annulus of unit tangent vectors on $\gamma_g$ pointing
to the same side of $S^2\backslash \gamma_g$, and $F_{g}$ be
the Birkhoff annulus  map on $A_g$ with respect to the geodesic flow $\phi_t^g$,
see \S \ref{sec-bir}.
Let $n\ge 1$, ${P}_n(F_g)$ be the set of points in $A_g$ that is fixed by $F_g^n$.
Note that ${P}_n(F_g)$ is always a closed (may not be compact) subset of $A_g$, since
the return function $t(\cdot)$ is bounded and bounded away from zero.

Let $(x,v)\in {P}_n(F_g)$. Then $(x,v)$ is said to be {\it nondegenerate under $F^n_g$}
if $\text{Tr}(D_{(x,v)}F^n_g)\neq 2$.
Note that the minimal period $m$ of $(x,v)$ under $F_g$ may be
smaller than $n$. In this case,
we set $k=\frac{n}{m}$, let $\gamma$ be the closed geodesic with initial
condition $(x,v)$,
and $\gamma^k$ be the geodesic that repeats itself $k$ times.
Then the above nondegeneracy condition is equivalent to that
$\text{Tr}(\gamma^k)\neq 2$, where $k=\frac{n}{m}$.

Let $\cU^r\subset \cG^r_+$ be the open and dense subset given by
Proposition \ref{assign},
$\cU^r_n\subset \cU^r$ be the subset of Riemannian metrics
such that $\text{Tr}(\gamma_g^n)\neq 2$
and each periodic point in ${P}_n(F_g)$ is nondegenerate under $F^n_g$.
The following  Proposition \ref{finite}  can be viewed
as the reformulation of Bumpy Metric Theorem in terms of its Birkhoff annulus maps.
This formulation is the key
to prove the existence of homoclinic intersections in Sect. \ref{sec:homo}.
\begin{pro}\label{finite}
Let $n\ge 1$, and  $\cU^r_n$ be the subset given above.
Then the following statements hold:
\begin{enumerate}
\item for each $g\in\cU^r_n$, ${P}_n(F_g)$ is a finite subset,

\item the map $g\in\cU_n^r\mapsto {P}_n(F_g)$ varies continuously.
\end{enumerate}
\end{pro}
\begin{proof}
(1). Let $g\in\cU^r_n$. Suppose on the contrary that there are infinitely many periodic
points in ${P}_n(F_g)$. Then we pick a sequence of points from ${P}_n(F_g)$,
say $\{(x_k,v_k):k\ge 1\}$, such
that all the points are mutually different and
they converge to some point $(x_\ast,v_\ast)$.
\begin{enumerate}
\item[1a).] $v_\ast\neq \pm\dot\gamma(x_\ast)$: then $(x_\ast,v_\ast)\in {P}_n(F_g)$.
Here we use the fact that the return time function $t(x,v)$ is bounded.
This point $(x_\ast,v_\ast)$  must be degenerate under $F^n_g$,
since the vector
$\ds \lim_{k\to\infty}\frac{(x_k,v_k)-(x_\ast,v_\ast)}{\|(x_k,v_k)-(x_\ast,v_\ast)\|}$
exists  (passing to a subsequence if necessary)
and is invariant under $DF^n_g$. This contradicts the choice of $g\in\cU^r_n$.

\item[1b).]  $v_\ast=\pm\dot\gamma_g(x_\ast)$: the geodesic $\gamma_k(t)=\exp_{x_k}(tv_k)$
converges to the geodesic $\gamma_g$, which implies that $\gamma_g^n$,
the $n$-iterate of $\gamma_g$, is degenerate under the return map $\phi_{nT}$.
 This also contradicts the choice of $g\in\cU^r_n$.
\end{enumerate}

(2). Note that the nondegenerate closed geodesics persist under perturbations,
and the nondegeneracy under $F^n_g$ is an open condition.
Therefore $g\in\cU_n^r\mapsto {P}_n(F_g)$ is lower semi-continuous.
The upper semi-continuity follows from the continuous dependence of
$\gamma_g$ and $F_g$ on $g\in \cU^r_n\subset \cU^r$, see Proposition
\ref{assign}. Putting them together, we prove the continuity of  ${P}_n(F_g)$.
\end{proof}

\begin{pro}\label{opendense}
The set $\cU^r_n$ contains an open and dense subset of $\cG^r_+$.
\end{pro}
\begin{proof}
The openness of $\cU^r_n$ follows from the finiteness
and continuity dependence of $P_n(F_g)$ in Proposition \ref{finite},
while the denseness follows from Bumpy Metric Theorem.
\end{proof}

Let $\gamma$ be a hyperbolic closed geodesic.
Then $W^s(\gamma)\backslash \gamma$
has two components, say $W^s_{\pm}(\gamma)$.
Similarly we define $W^u_{\pm}(\gamma)$.

\begin{definition}\label{complete}
Let $P_n(F_{g})=\{(x_i,v_i):1\le i\le k\}$ for some $k=k(n)$,
$\gamma_i$ be the closed geodesic on $S^2$ with initial condition $(x_i,v_i)$,
To include the special closed geodesic $\gamma_g$, we re-label it by $\gamma_0$.
Putting them together, we denote the set by $\Gamma_n(g)=\{\gamma_i:0\le i\le k\}$.
\end{definition}

Let $\cV^r_n\subset \cU^r_n$ be the subset of metrics $g$ such that
for any pair of stable and unstable components of
two hyperbolic closed geodesics in $\Gamma_n(g)$,
either they do not intersect, or they admit some transverse intersections.
\begin{pro}\label{opendense1}
The set $\cV^r_n$ contains an open and dense subset of $\cG^r_+$.
\end{pro}
\begin{proof}
Let $g_0\in  \cU^r_n$, $P_n(F_0)$ be the set of points fixed by $F^n_0$.
We can label them by $\{(x_1,v_1),\dots,(x_k,v_k)\}$,
all being nondegenerate under $F^n_0$.
According to Proposition \ref{finite}, we can find
a small neighborhood $\cV$ of $g_0$ such that
$P_n(F_{g})=\{(x_1(g),v_1(g)),\dots,(x_k(g),v_k(g))\}$,
where $(x_i(g),v_i(g))$ is the continuation of $(x_i,v_i)$, $1\le i\le k$.
Let $\gamma_i$ be the closed geodesic on $(S^2,g_0)$ with initial condition $(x_i,v_i)$,
and  $\gamma_i(g)$ be the closed geodesic on $(S^2,g)$
with initial condition $(x_i(g),v_i(g))$.
To include the special closed geodesic $\gamma_g$, we re-label it by $\gamma_0(g)$.
Putting them together as in Definition \ref{complete}, we denote the set by
$\Gamma_n(g)$.

Let $i,j\in\{0,\dots,k\}$ be two indices such that $\gamma_i$ and $\gamma_j$ are hyperbolic,
and $\alpha,\beta\in\{+,-\}$ indicate the components
of the stable and unstable manifolds we could pick.
Note that we have added the index $0$ corresponding to the geodesic $\gamma_g$.
Let $\cV_{ij\alpha\beta}$ be the subset of metrics $g$ in $\cV$ such that
either $W^{s}_\alpha(\gamma_i(g))\cap W^{u}_\beta(\gamma_j(g))=\emptyset$,
or $W^{s}_\alpha(\gamma_i(g))$ and $W^{u}_\beta(\gamma_j(g))$
admit some transverse intersections.
It suffices to show that $\cV_{ij\alpha\beta}$ contains an open and dense subset
in $\cV$.

We partition $\cV$ into two parts $\cV_1\cup\cV_2$, where
\begin{enumerate}
\item $g\in\cV_1$ if
$W^{s}_\alpha(\gamma_i(\hat g))\cap W^{u}_\beta(\gamma_j(\hat g))=\emptyset$
for any $\hat g$ sufficiently close to $g$;

\item $g\in\cV_2$ if
$W^{s}_\alpha(\gamma_i(g_l))\cap W^{u}_\beta(\gamma_j(g_l))\neq\emptyset$
for some sequence $g_l\to g$ as $l\to\infty$.
\end{enumerate}
It is clear that $\cV_1$ is open (may be empty) and $\cV_1\subset \cV_{ij\alpha\beta}$.
For each $g\in\cV_2$, let $g_l$, $l\ge 1$
 be the sequence given as in the above definition of $\cV_2$.
If $W^{s}_\alpha(\gamma_i(g_l))$ and $W^{u}_\beta(\gamma_j(g_l))$
admits a transverse intersection for infinitely many $l$, then
there exists a small neighborhood of $g_l$ that is contained in $\cV_{ij\alpha\beta}$.
If $W^{s}_\alpha(\gamma_i(g_l))$ and $W^{u}_\beta(\gamma_j(g_l))$
intersect non-transversely at some point, say $z=(x,v)$,
then we apply Proposition \ref{don1}
to find a small perturbation of $g_l$ such that
$W^{s}_\alpha(\gamma_i(\hat g_l))$ and $W^{u}_\beta(\gamma_j(\hat g_l))$
admits a transverse intersection. Then we use this new sequence $\hat g_l$.
This completes the proof.
\end{proof}

Let $\cR^r_{KS}=\bigcap_{n\ge 1}\cV^r_n$, which contains a residual subset of $\cG^r_+$.
\begin{thm}\label{KS}
There is a residual subset $\cR^r_{KS}$ of $\cG^r_+$, such that for any $g\in\cR^r_{KS}$,
the geodesic flow $\phi_t$ on the unit tangent bundle  $M_g\subset TS^2$ satisfies:
\begin{enumerate}
\item every closed geodesic is either hyperbolic or irrationally elliptic.

\item if $\gamma$ and $\eta$ are hyperbolic, then for any component of
$W^{s}_{\pm}(\gamma)$ and of $W^{u}_{\pm}(\eta)$, either they do not intersect,
or they admit some transverse intersections.
\end{enumerate}
\end{thm}

Let $\cV^\infty_n=\left(\bigcup_{r\ge 2}\cV^r_n\right)\cap \cG^\infty_+$, which
contains an open and dense subset of $\cG^\infty_+$.
Let $\cR^\infty_{KS}=\bigcap_{n\ge 1}\cV^\infty_n$, which
contains a residual subset of $\cG^\infty_+$.
Therefore, the conclusions of Theorem \ref{KS} also hold for $r=\infty$.

\begin{remark}
The above theorem does not claim
that the two components $W^{s}_{\pm}(\gamma)$ and $W^{u}_{\pm}(\eta)$
are transverse, since we do not try to remove all non-transverse intersections.
More crucially, the above theorem does not specify
when the two components $W^{s}_{\pm}(\gamma)$ and  $W^{u}_{\pm}(\eta)$
have nontrivial intersections. In the next section, we will
show the existence of homoclinic intersections
for each hyperbolic closed geodesic on a generic convex sphere.
\end{remark}

\section{Homoclinic intersections for hyperbolic closed geodesics}
\label{sec:homo}

In this section, we study the existence of homoclinic intersections
of hyperbolic closed geodesics of the geodesic flow on the unit tangent bundle
of a convex sphere. More precisely, let $\cG^r_+$ be the set of convex Riemannian metrics
on $S^2$, that is, the set of metrics with positive Gauss curvature.
There is an open and dense subset $\cU^r\subset \cG^r_+$ (see Proposition \ref{assign})
such that for each $g\in \cU^r$, there is exactly
one simple closed geodesic $\gamma_g$ of shortest length.
Let $A_g\subset S_g$ be the Birkhoff annulus and $F_g$ be the Birkhoff annulus map
of the geodesic flow with respect to $A_g$,
$P_n(F_g)$ be the set of points in $A_g$ fixed by $F^n_g$, and
$\Gamma_n(g)$ be the set of closed geodesics corresponding to $P_n(F_g)$
with one additional $\gamma_g$ (see Definition \ref{complete}).
Our main result in this section is the following.
\begin{pro}\label{homo1n}
There is an open and dense subset $\cW_n^r\subset \cG^r_+$
such that for each $g\in \cW_n^r$, for each hyperbolic closed geodesic
$\gamma\in\Gamma_n(g)$,
there exist some transverse homoclinic intersections for $\gamma$.
\end{pro}

It suffices to show that such $\cW^r_n$ is open and dense in $\cV^r_n$
(see Proposition \ref{opendense1} and the paragraph above it
for the definition of the  subset $\cV_n^r$).
Note that $P_n(F_g)$ and $\Gamma_n(g)$
are finite and vary continuously over $g\in\cV_n^r$.
Moreover, transverse homoclinic intersections, once created,
persist under small perturbations.
Therefore the set $\cW_n^r$ is  open in $\cV_n^r$.
So it suffices to prove the $C^r$ denseness of $\cW_n^r$ in $\cV_n^r$.

\subsection{Nonlinear stability of elliptic periodic points}\label{sec-her}
Let $f$ be a $C^\infty$ diffeomorphism on a surface $S$
that preserves a smooth measure on $S$. Let  $p$ be a fixed point of $f$.
Then $p$ is said to be {\it nonlinearly stable},
if there exist a sequence of $f$-invariant nesting closed disks
$D_n$, $n\ge 1$  such that  $p\in D_{n+1}\subset D^o_n$,
$\bigcap_n D_n =\{p\}$ and $f|_{\partial D_n}$ is transitive.
If a fixed point is not  nonlinearly stable, then we say it is
{\it nonlinearly unstable}.

Recall that a real number $\rho$
is said to be {\it Diophantine}, if there exist two positive numbers $c,\tau$
such that
\begin{equation}
\left|\rho-\frac{m}{n}\right|\ge \frac{c}{|n|^{2+\tau}},
\text{ for all rational numbers }\frac{m}{n}.
\end{equation}
Then an elliptic fixed point $p$ of $f$
is said to have Diophantine rotation number,
if the tangent map $D_pf$ is conjugate to a rotation
$R_\rho=\begin{bmatrix}\cos2\pi\rho & -\sin2\pi\rho \\
\sin2\pi\rho & \cos2\pi\rho \end{bmatrix}$
for some Diophantine number $\rho$.

The following is called Herman's {\it Last
Geometric Theorem}, which states that  an elliptic fixed point with Diophantine
rotation number is nonlinearly stable.
See \cite{FK} for the history and a complete proof of this theorem.

\noindent{\bf Herman's  Last Geometric Theorem.}
{\it Let $f\in\mathrm{Diff}^\infty_\mu(S)$ and $p$ be an elliptic fixed point of $f$.
If the rotation number of $p$ is Diophantine, then $p$ is nonlinearly stable.}

\begin{remark}
An elliptic fixed point may be nonlinearly {\it unstable}
if its rotation number $\rho$ is {\it not} Diophantine.
For example, Anosov and Katok constructed in \cite{AK}
a weak mixing area-preserving diffeomorphism $f$
 on  the unit disk $\mathbb{D}\subset \bR^2$
for which the origin $o\in \mathbb{D}$ is an elliptic fixed point.
In this example, the fixed point $o$ is nonlinearly {\it unstable}
and its rotation number is {\it not} Diophantine.
\end{remark}

\vskip.05in

By taking a cross-section and considering the Poincar\'e map, we see that
the same result holds for elliptic closed geodesics of the geodesic flows on $(S^2,g)$.

We will use Herman's LGT to prove the denseness of $\cW^r_n$ in $\cV^r_n$.
\begin{pro}\label{dense}
There is a $C^r$ dense subset $\cD_n\subset \cV_n^r\cap \cG^\infty_+$
such that the following hold for each $g\in \cD_n$:
\begin{enumerate}
\item every elliptic periodic point in $P_n(F_g)$ is nonlinearly stable;

\item the shortest geodesic $\gamma_g$ is either hyperbolic or nonlinearly stable.
\end{enumerate}
\end{pro}
\begin{proof}
Let $g\in\cV_n^r\cap \cG^\infty_+$ be a $C^\infty$ smooth metric,
$\gamma_g$ be the shortest simple closed geodesic
and $F_g$ be the Birkhoff annulus map of the geodesic flow $\phi_t$
on the Birkhoff annulus $A_g=\gamma_g\times (0,\pi)$.
Let $P_n(F_g)$ be the set of points in $A_g$ fixed by $F_g^n$,
and label them as $\{(x_i,v_i)\}_{i=1}^{k_n}$. Let $\gamma_i$ be the closed geodesic
with initial condition $(x_i,v_i)$ for each $i=1,\cdots,k_n$
and let $\gamma_0=\gamma_g$.
Clearly they are all nondegenerate. For each $j=0,\cdots,k_n$, we pick a simple
point $x_j\in\gamma_j$ in the sense that the union $\bigcup_{i=0}^{k_n}\gamma_i$
covers $x_j$ only once. Then we make a $C^\infty$-smooth and $C^r$-small
perturbation of the metric $g$ in a small neighborhood
$U_j$ of $x_j$, say the new metric $\hat g$, such that
every elliptic closed geodesic among those $\{\gamma_i(\hat g):0\le i\le {k_n}\}$
has Diophantine rotation number.
Then Herman's LGT guarantees that all the elliptic closed geodesics
among $\{\gamma_i(\hat g):0\le i\le {k_n}\}$ are nonlinearly stable,
and hence $\hat g\in \cD_n$.
Therefore, $\cD_n\subset \cG^\infty_+$ and is dense in $\cV^r_n$.
This completes the proof.
\end{proof}

After we are done with the elliptic ones, let's move on to study the hyperbolic
periodic points in $P_n(F_g)$.
Although each hyperbolic periodic point $(x,v)\in P_n(F_g)$  is fixed by $F^n_g$,
the two branches of the stable (and unstable) manifolds of $(x,v)$  may
be switched by $F^n_g$.
A simple trick we use here is to double the iterates: the map $f=F^{2n}_g$
not only fixes each point in $P_n(F_g)$,
but also fixes each branch of the invariant manifolds
of the hyperbolic periodic points in $P_n(F_g)$.
In the following we will use this pairing
between $2n$-th iterate
and the periodic orbits of period $n$.
Let $\cD_n$, $n\ge 1$ be the sequence of dense subsets of $\cV_n^r$ given by
Proposition \ref{dense}.
The following lemma is proved by applying
the {\it Prime End Compactification} method developed by Mather in
\cite{Mat81,Mat82}, since the topology of the
manifold $A_g$ is very simple. See also \cite{Oli1,FrLC}.
\begin{lem}\label{PnF}
Let $g\in \cD_{2n}$. Then every hyperbolic periodic point in $P_n(F_g)$
admits some transverse homoclinic intersections.
\end{lem}
As we have pointed out,
the index $2n$ of $\cD_{2n}$ is {\it  not} a misprint.
\begin{proof}
Let $g\in \cD_{2n}$,  $\gamma_g$ be the shortest simple closed geodesic
and $F_g$ be the Birkhoff annulus map of the geodesic flow $\phi_t$
on the Birkhoff annulus $A_g=\gamma_g\times (0,\pi)$.
For simplicity we denote $f=F^{2n}_g$.
Note that all elliptic fixed points of $f$ are nonlinearly stable,
and any two branches of the stable and unstable manifolds of
the hyperbolic fixed points are either disjoint or admit some transverse intersections.

Let $P_n(F_g)$ be the set of points in $A_g$ fixed by $F_g^n$,
and $z=(x,v)\in P_n(F_g)$ be a hyperbolic periodic point.
Then $z$ is a hyperbolic fixed point of $f$ and all four branches
of the stable and unstable manifolds are also fixed by $f$.
Then Mather's theorem \cite[Theorem 5.2]{Mat81} implies that all four
branches $W^{s,u}_{\pm}(z,f)$ are recurrent and they have the same closure:
$\overline{W^{\sigma}_\alpha(z,f)}=E(z)$ for each $\sigma\in\{s,u\}$
and each $\alpha\in\{+,-\}$, see also \cite[Corollary 3.4]{XZ}.
Then it is standard to prove the existence of homoclinic intersections assuming the 
recurrence of the stable and unstable manifolds,
see \cite[Theorem 4.3]{XZ} for more details.
For completeness, we give some explanation
below to illustrate the geometric picture about the construction.

First, we set up a local coordinate system $\chi:\mathbb{R}^2\to U$
around a small neighborhood $U\subset  A_g$ of $z$,
such that $\chi(0)=z$,
the local stable manifold of $z$ leaves $0$ along the $x$-axis,
and the local unstable manifold of $z$ leaves $0$ along the $y$-axis.
Then the local stable and unstable manifolds of $z$ divide $U$ into four
quadrants, say $Q_i$, $1\le i\le 4$.
Since the stable branch $W^s_+(z)$ is recurrent, it will accumulate
on itself through the neighboring quadrants $Q_1$ and/or $Q_4$.
Without loss of generality
we assume $W^s_+(z)$ accumulates in the first quadrant $Q_1$.
Then the unstable branch $W^u_+(z)$ will also accumulate
in $Q_1$, since its closure contains $W^s_+(z)$ (by Mather's result).

Pick a small number $\epsilon>0$,
and consider the subset $S_\epsilon=\{(x,y)\in U:0< x,y\le 1, xy\le \epsilon \}$
in $Q_1$. Let $q$ be the first moment on $W^s_+(z)$ that hits $S_\epsilon$,
and $\gamma^s$ be the closed curve that starts from $z$, moves along
 $W^s_+(z)$  to the point $q$, and then along the segment
 from $q$ to $z$. Note that $\gamma^s$ is a simple closed curve,
 since $q$ is the first intersection of $W^s_+(z)$ with $S_\epsilon$.
Similarly, let $\hat q$ be the first moment on $W^u_+(z)$ that hits $S_\epsilon$,
and construct the curve $\gamma^u$.
See Fig.~\ref{adjacent}.

\begin{figure}[h]
\begin{overpic}[width=60mm]{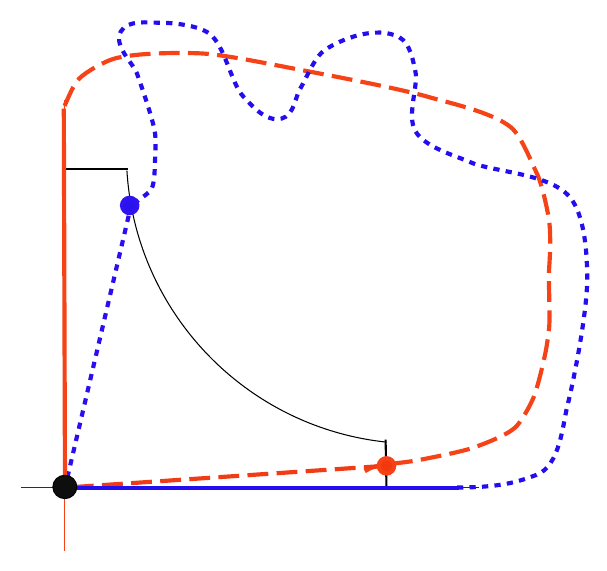}
\put(13,8){$z$}
\put(33,7){$W^s_+(z)$}
\put(-10,35){$W^u_+(z)$}
\put(83,35){$\gamma^u$}
\put(98,35){$\gamma^s$}
\put(20,25){$S_{\epsilon}$}
\put(25,60){$q$}
\put(65,21){$\hat q$}
\end{overpic}
\caption{The closing curves $\gamma^s$ (blue)
and $\gamma^u$ (red).}
\label{adjacent}
\end{figure}

These two simple closed curves $\gamma^{s,u}$
cross each other once at the fixed point $z$.
On the other hand, the algebraic intersection number of two simple closed curves
on an annulus $A_g$ must be zero. Therefore, the two curves must intersect
somewhere else, and that intersection is a homoclinic intersection.
This completes the proof.
\end{proof}

Comparing the set $P_n(F_g)$ with $\Gamma_n(g)$,
we see that there is only one geodesic left from Lemma \ref{PnF}:
the geodesic $\gamma_g$ for each $g\in\cD_{2n}$.
\begin{remark}
Poincar\'e \cite{Poi05} conjectured
that there exists a non-hyperbolic simple closed geodesic on each
convex surface. If this were true, one may try to argue that $\gamma_g$
is non-hyperbolic. However, Grjuntal' \cite{Gr} constructed an open set of
convex metrics on $S^2$ such that every simple closed geodesic is hyperbolic.
Elliptic closed geodesic does exist $C^2$ densely \cite{CO},
just that it may not be simple.
\end{remark}

Let  us return to the proof of Proposition \ref{homo1n}.
If $\gamma_g$ is elliptic, then we are done.
If  $\gamma_g$ is hyperbolic,
we will prove that the choice of $g\in\cD_{2n}$ automatically implies
the existence of homoclinic intersections for $\gamma_g$.
Note that no perturbation is needed in the following proof.
\begin{lem}\label{miss}
Let $g\in\cD_{2n}$ such that $\gamma_g$ is a hyperbolic closed geodesic.
Then  $\gamma_g$ admits some transverse homoclinic intersections.
\end{lem}
\begin{proof}
Let $\eta_g$ be a second simple closed geodesic on $S^2$, whose existence is given by
Lyusternik-Shnirel'man theorem \cite{LySh,Tai}.
Let $\hat F_g$ be the new Poincar\'e map induced on the new
Birkhoff annulus $\hat A=A_{\eta_g}$. Pick  $y_0\in\gamma_g\cap \eta_g$
and $(y_0,u_0)$ be the point in $\hat A$ that generates $\gamma_g$.
For the second iterate $\hat F_g^2$, the following hold:
\begin{enumerate}
\item every fixed point $(y,u)$ of $\hat F_g^2$ is also fixed by $F_g^{2n}$,
since it corresponds to a closed geodesic
of uniformly bounded length;

\item every elliptic fixed point of $\hat F_g^2$ is nonlinearly stable, since
the corresponding closed geodesic is (see Proposition \ref{dense});

\item any two branches of the stable and unstable manifolds
of two hyperbolic fixed points of $\hat F_g^2$ either don't intersect
or admit some transverse intersections, since
the corresponding components of the closed geodesics are
(see Proposition \ref{opendense1});

\item $(y_0,u_0)$ is a  hyperbolic fixed point of $\hat F_g$ and all four branches
are fixed by $\hat F_g^2$ (since $\gamma_g$ is a simple closed geodesic).
\end{enumerate}
Then the same argument given in Lemma \ref{PnF} shows the
existence of transverse intersections between the stable and unstable manifolds
of the fixed point $(y_0,u_0)$ for any $g\in\cD_{2n}$.
This completes the proof.
\end{proof}

\begin{proof}[Proof of Proposition \ref{homo1n}]
Combining Lemma \ref{PnF} and \ref{miss}, we see that for any $g\in \cD_{2n}$,
there exist transverse homoclinic intersections for any closed geodesic in
$\Gamma_n(g)$. This implies that $\cD_{2n}\subset\cW^r_n$.
Therefore $\cW^r_n$ must be dense in $\cV^r_n$,
since $\cD_{2n}$ is dense in $\cV^r_{2n}$, which is also dense in $\cG^r_+$.
This completes the proof of Proposition \ref{homo1n}.
\end{proof}

\begin{proof}[Proof of Theorem \ref{main}]
Let $\cW^r_n$ be the open and dense subset of $\cG^r_+$ given by Proposition \ref{homo1n}.
Then the set $ \cR^r_{h}=\bigcap_{n\ge 1}\cW^r_n$
contains a residual subset of $\cG^r_+$. Let $g\in \cR^r_{h}$.
Then each hyperbolic closed geodesic of the geodesic flow on
the unit tangent bundle $M_g$ admits some transverse homoclinic intersections.
Combining with Theorem \ref{KS},
we complete the proof of Theorem \ref{main}.
\end{proof}

\section*{Acknowledgments}
This research is supported in part by National Science Foundation.
The authors are very grateful to the anonymous referee for many useful
comments and suggestions,
which helped them to improve the presentation of the paper significantly.

\end{document}